\documentclass[12pt]{amsart}
\usepackage{amssymb, hyperref, fullpage, xcolor}

\newtheorem{theorem}{Theorem}[section]
\newtheorem{lemma}[theorem]{Lemma}

\newtheorem{definition}[theorem]{Definition}

\theoremstyle{remark}
\newtheorem{remark}[theorem]{Remark}
\newtheorem{example}[theorem]{Example}

\numberwithin{equation}{section}

\newcommand{\calO}{\ensuremath{\mathcal{O}}}

\newcommand{\bbQ}{\ensuremath{\mathbb{Q}}}
\newcommand{\bbQbar}{\ensuremath{{\overline\bbQ}}}

\newcommand{\bbQpbar}{{\ensuremath{\overline{\mathbb{Q}}_p}}}
\newcommand{\bbQl}{\ensuremath{{\mathbb{Q}_\ell}}}

\newcommand{\End}{\ensuremath{\mathrm{End}}}
\newcommand{\Fr}{\ensuremath{\mathrm{Fr}}}
\newcommand{\Frss}{\ensuremath{\mathrm{Fr\text{-}ss}}}

\newcommand{\Gal}{\ensuremath{\mathrm{Gal}}}

\newcommand{\tr}{\ensuremath{\mathrm{tr}}}

\newcommand{\WD}{\ensuremath{\mathrm{WD}}}
\newcommand{\wild}{\ensuremath{\mathrm{wild}}}

\newcommand{\xra}{\xrightarrow}

\newcommand{\bbZ}{\ensuremath{\mathbb{Z}}}
\newcommand{\bbZp}{\ensuremath{{\mathbb{Z}_p}}}

\newcommand{\gln}{\ensuremath{\operatorname{GL}}}

\newcommand{\Sp}{\ensuremath{\mathrm{Sp}}}

\newcommand{\mo}{{-1}}

\begin{document}
\title{Distinguishing pure representations by normalized traces}

\author{Manish Kumar Pandey}
\address{Harish-Chandra Research Insititute HBNI, Chhatnag Road, Jhunsi, Allahabad 211019, India}
\curraddr{}
\email{manishpandey@hri.res.in}
\thanks{}

\author{Sudhir Pujahari}
\address{Harish-Chandra Research Insititute HBNI, Chhatnag Road, Jhunsi, Allahabad 211019, India}
\curraddr{}
\email{sudhirpujahari@hri.res.in}
\thanks{}

\author{Jyoti Prakash Saha}
\address{Department of Mathematics, Ben-Gurion University of the Negev, Be'er Sheva 8410501, Israel}
\curraddr{}
\email{jyotipra@post.bgu.ac.il}
\thanks{}



\begin{abstract}
Given two pure representations of the absolute Galois group of an $\ell$-adic number field with coefficients in $\bbQpbar$ (with $\ell\neq p$), we show that the Frobenius-semisimplifications of the associated Weil--Deligne representations are twists of each other by an integral power of certain unramified character if they have equal normalized traces. This is an analogue of a recent result of Patankar and Rajan in the context of local Galois representations. 
\end{abstract}

\maketitle

\section{Introduction}

\subsection{Motivation}
The study of representations of the absolute Galois groups of number fields and local fields is a central theme in number theory. As a consequence of the Chebotarev density theorem, it follows that a continuous representation of the absolute Galois group of a number field (which are unramified almost everywhere and have coefficients in $\bbQpbar$ for instance) can be characterized up to semisimplification by the traces of the Frobenius elements at the unramified places. This statement for the Galois representations associated with normalized Hecke eigen new cusp forms can be seen as an analogue of the strong multiplicity one theorem. In \cite[Theorem 2 (iii)]{RajanStrongMultOne}, Rajan proved under suitable assumptions that two continuous semisimple representations of the absolute Galois group of a number field are isomorphic up to a twist by a Dirichlet character if they have equal traces at the Frobenius elements at places varying in a set of positive upper density. This provides a refinement of the strong multiplicity one theorem for Hecke eigen cusp forms \cite[Corollary 1]{RajanStrongMultOne}. 

From a recent result of Kulkarni, Patankar and Rajan \cite{KulkarniPatankarRajan}, it follows that two elliptic curves $E_1, E_2$ defined over a number field $F$ are isogeneous over an extension of $F$ if their Frobenius fields are equal at places of positive upper density and one of $E_1, E_2$ is without complex multiplication. This inspired a result of Murty and Pujahari \cite[Theorem 1.1]{MurtyPujahari}, which states that two normalized Hecke eigen cusp forms are twists of each other by some Dirichlet character if at least one of them is without complex multiplication and both of them have equal normalized Hecke eigenvalues at primes varying in a set of positive upper density. This result can be restated in terms of their associated Galois representations, which are predicted to be pure by the Ramanujan conjecture. Since many Galois representations of arithmetic interest are pure (as predicted by the weight-monodromy conjecture \cite[Conjecture 3.9]{IllusieMonodromieLocale}), we may expect an analogue of the result of \cite{MurtyPujahari} for pure Galois representations. Indeed, such a result is established by Patankar and Rajan \cite[Theorem 2]{PatankarRajan}. Under appropriate assumptions, they proved that if $\rho_1, \rho_2$ are two pure representations of the absolute Galois group of a number field with coefficients in $\bbQpbar$ and the actions of the Frobenius element $\Fr_v$ on $\rho_1,\rho_2$ have equal normalized traces for $v$ varying in a set of places of positive upper density, then $\rho_1$ is a twist of $\rho_2$ by the product of a power of the $p$-adic cyclotomic character and a finite order character. Its proof relies on \cite[Theorem 2 (iii)]{RajanStrongMultOne}. 

In this article, we investigate an analogue of \cite[Theorem 2]{PatankarRajan} for representations of the absolute Galois groups of $\ell$-adic number fields with coefficients in $\bbQpbar$ for $\ell\neq p$. 

\subsection{Result obtained}
Let $p, \ell$ be distinct primes. Let $K$ be a finite extension of $\bbQl$. We denote by $q$ the cardinality of the residue field of the ring of integers of $K$. Let $W_K$ denote the Weil group of $K$. 
Given a continuous representation $\rho$ of the absolute Galois group $\Gal(\overline K/K)$ of $K$ with coefficients in $\bbQpbar$, the Frobenius-semisimplification of its Weil--Deligne parametrization is denoted by $\WD(\rho)^\Frss$ (see \S \ref{Sec: preliminaries}). 
We obtain the following result. 

\begin{theorem}
\label{Theorem Introduction}
Let $\rho_1, \rho_2$ be pure representations of the absolute Galois group $\Gal(\overline K/K)$ with coefficients in $\bbQpbar$. If $\rho_1, \rho_2$ have equal normalized traces, then $\WD(\rho_1)^\Frss$ is isomorphic to $\psi^w \otimes \WD(\rho_2)^\Frss$ where $\psi:W_K \to \bbQbar_p^\times$ denotes the unramified character sending the geometric Frobenius element to $q^{1/2}$ and $w$ denotes the difference of the weights of $\rho_1$ and $\rho_2$.
\end{theorem}

The above theorem follows from a similar result about Weil--Deligne representations, proved in Theorem \ref{Theorem WD}. The proof of Theorem \ref{Theorem WD} is based on the observation that the knowledge of a highest and a lowest weight irreducible summand of a Frobenius-semisimple pure Weil--Deligne representation (when thought of as a representation of the Weil group by forgetting the monodromy) determines a `large part' (i.e., an indecomposable summand having monodromy with highest degree of nilpotency) of the pure representation and as a consequence, a Frobenius-semisimple pure representation can be well-understood from its trace via an induction argument. This observation is first used in \cite[Chapter 1]{Thesis} (to the best of our knowledge) and subsequently used in the proofs of \cite[Theorem 3.1, Lemma 4.2, Proposition 4.1, Theorem 4.3]{BigPuritySubAIF}. Moreover, Theorem \ref{Theorem WD} can be recovered from \cite[Theorem 4.3]{BigPuritySubAIF} (as outlined in Remark \ref{Remark}), which studies pure specializations of pseudorepresentations. However, our proof of Theorem \ref{Theorem WD} does not use the notion of pseudorepresentations. 

\subsection{Acknowledgements}
The first author would like to thank Prof.\,B.\,Ramakrishnan for his encouragement, during the work of this paper author is supported by the Infosys scholarship. The third author would like to thank Prof.\,C.\,S.\,Dalawat for an invitation to the Harish-Chandra Research Institute during June, 2017 when this work was initiated. He is grateful to the institute for providing a warm hospitality. During the final stage of the work, he was supported by a postdoctoral fellowship at the Ben-Gurion University of the Negev, offered by the Israel Science Foundation Grant number 87590011 of Prof.\,Ishai Dan-Cohen. 

\section{Preliminaries}
\label{Sec: preliminaries}
Let $\ell$ be a prime and $K$ denote a finite extension of $\bbQl$. Denote by $G_K$ the absolute Galois group $\Gal(\overline{K}/K)$ of $K$. Let $I_K$ (resp. $I_K^\wild$) denote the inertia (resp. wild inertia) subgroup of $G_K$. Let $\varpi$ denote a uniformizer of the ring of integers $\calO_K$ of $K$. Given a compatible system of roots of unity $\zeta = (\zeta_n)_{\ell\nmid n} $, there is an isomorphism 
$t_\zeta: I_K/I_K^\wild \xra{\sim} \prod_{\ell' \neq \ell} \bbZ_{\ell'}$
such that $\sigma (\varpi^{1/n}) = \zeta_n ^{t_\zeta(\sigma) } \varpi^{1/n}$ for any $\sigma\in I_K/I_K^\wild$. 
Fix a prime $p$ with $p\neq \ell$. Let $t_{\zeta, p}: I_K \to \bbZp$ denote the composition of the quotient map $I_K \to I_K/I_K^\wild$, the map $t_\zeta$ and the projection map $\prod_{\ell' \neq \ell} \bbZ_{\ell'} \to \bbZp$. The Weil group $W_K$ is defined as the inverse image of the subgroup generated by the geometric Frobenius element $\Fr$ under the projection map $G_K\to G_K/I_K$. We put the smallest topology on $W_K$ such that $I_K$ (with its usual topology) is open in it. Let $v_K: W_K \to \bbZ$ denote the group homomorphism which is trivial on inertia and sends $\Fr$ to $1$. Let $q$ denote the cardinality of the residue field of $\calO_K$. Henceforth we fix a lift $\varphi\in G_K$ of $\Fr$ and a square root $q^{1/2}$ of $q$ in $\bbQpbar$. The following result due to Grothendieck explains the action of $I_K$ on $p$-adic representations of $G_K$. 

\begin{theorem}
[Grothendieck]
\label{Thm: Grothendieck}
\cite[pp. 515--516]{SerreTate}
Let $\rho: G_K \to \gln(V)$ be a continuous representation of $G_K$ on a finite dimensional $\bbQpbar$-vector space $V$. Then there exists an open normal subgroup $U$ of $I_K$ and a unique nilpotent element $N \in \End_\bbQpbar(V)$ such that 
$$\rho(u ) = \exp (t_{\zeta, p} (u) N) \quad \text{ for all } u \in U.$$
The nilpotent operator $N$ is called the \textnormal{monodromy} of $\rho$. 
\end{theorem}

\begin{definition}
\label{WD definition}
A \textnormal{Weil--Deligne representation} of $W_K$ on a finite dimensional $\bbQpbar$-vector space $V$ is a pair $(r, N)$ where $r$ is a $\bbQpbar$-linear representation of $W_K$ on $V$ with open kernel and $N$ is a nilpotent operator (called the \textnormal{monodromy}) on $V$ such that 
$$r(\sigma) N r(\sigma)^\mo = q^{-v_K(\sigma)}N$$
for all $\sigma\in W_K$ \cite[8.4.1]{DeligneConstantesDesEquationsFunctional}. 
The \textnormal{monodromy filtration} of $(r,N)$ is the filtration $M_\bullet$ on $V$ defined by 
$$M_k = \sum_{i+j=k} \ker N^{i+1} \cap N^{-j}V$$
for $k\in \bbZ$ \cite[1.7.2]{DeligneWeil2}. 
\end{definition}

For a representation $\rho$ as in Theorem \ref{Thm: Grothendieck} with monodromy $N$, define its {\it Weil--Deligne parametrization} $\WD(\rho)$ as the Weil--Deligne representation $(r, N)$ where $r$ is a representation of $W_K$ on $V$ by 
$$r(g) = \rho(g) \exp (-t_{\zeta, p}(\varphi^{-v_K(g)}g)N)$$
for all $g\in W_K$ \cite[8.4.2]{DeligneConstantesDesEquationsFunctional}. 

\begin{example}
Let $\chi : W_K\to \bbQbar_p^\times$ denote the unramified character which sends $\varphi$ to $q^\mo$. Given a representation $r$ of $W_K$ on a finite dimensional $\bbQpbar$-vector space $V$ with open kernel, and a positive integer $t$, define the {\it special representation} $\Sp_t(r)$ as the Weil--Deligne representation of $W_K$ with the direct sum 
$$V_{t-1} \oplus \cdots \oplus V_0,
\quad V_i := V,
$$
of $t$-copies of $V$ as the underlying space, on which $W_K$ acts by 
$$r \chi^{t-1} \oplus r \chi^{t-2}\oplus \cdots \oplus r \chi \oplus r,$$
and the monodromy induces the identity map from $V_i$ to $V_{i+1}$ for each $0\leq i <t-1$ and vanishes on the summand $V_{t-1}$. 
\end{example}

Given a Weil--Deligne representation $(r, N)$ as in Definition \ref{WD definition}, write $r(\varphi) = r(\varphi)^{ss} r(\varphi)^u = r(\varphi)^{u} r(\varphi)^{ss}$ where $r(\varphi)^{ss}$ (resp. $r(\varphi)^u$) is a semisimple (resp. unipotent) operator. Define 
$$\widetilde r (g)  = r(g) (r(\varphi)^u)^{-v_K(g)} \quad\text{ for all } g \in W_K.$$
Then by \cite[8.5]{DeligneConstantesDesEquationsFunctional}, the pair $(\widetilde r, N)$ is a Weil--Deligne representation of $W_K$ on the underlying space of $(r, N)$. The pair is denoted by $(r, N)^\Frss$ and is called the {\it Frobenius--semisimplification} of $(r, N)$ \cite[D\'efinition 8.6]{DeligneConstantesDesEquationsFunctional}. 

\begin{definition}
A Weil--Deligne representation $(r, N)$ on a finite dimensional $\bbQpbar$-vector space $V$ is called \textnormal{pure of weight $w$} if the eigenvalues of $\varphi$ on the $i$-th grading of the monodromy filtration of $(r, N)$ are $q$-Weil numbers of weight $w+i$ for any $i\in \bbZ$. A continuous representation $\rho$ of $G_K$ on a finite dimensional $\bbQpbar$-space is called \textnormal{pure of weight $w$} if $\WD(\rho)$ is pure of weight $w$. 
\end{definition}

\begin{definition}
If $(r, N)$ is a pure Weil--Deligne representation of weight $w$, then its \textnormal{normalized trace} is defined as the trace of the representation $\psi^{-w}\otimes r$. 
If $\rho:G_K\to \gln_n(\bbQpbar)$ is a pure representation of weight $w$, then its \textnormal{normalized trace} is defined as the trace of the representation $\psi^{-w}\otimes (\rho|_{W_K})$ of $W_K$. 
\end{definition}

Note that the notion of normalized trace depends on the choice of the square root $q^{1/2}$ of $q$ in $\bbQpbar$. 

\begin{lemma}
\label{Lemma sum}
Suppose $(r, N)$ is a Weil--Deligne representation of $W_K$ with coefficients in $\bbQpbar$. Suppose $\sigma_1, \cdots, \sigma_k$ are irreducible Frobenius-semisimple pure representations of $W_K$ such that the sum of their traces is equal to the trace of $r$ and $\sigma_1$ has maximal weight among them. Assume further that the difference of the weights of a highest and lowest weight representation among $\sigma_1, \cdots, \sigma_k$ is $2(t-1)$ for a positive integer $t$. Then $(r, N)^\Frss$ is isomorphic to the direct sum $(r', N') \oplus \Sp_{t} (\sigma_1)$ for some pure Weil--Deligne representation $(r', N')$ of $W_K$. 
\end{lemma}

\begin{proof}
By \cite[Proposition 3.1.3(ii)]{DeligneFormesModulairesGL2}, there exist positive integers $t_1\leq \cdots \leq t_m$ and Frobenius-semisimple irreducible representations $r_1,\cdots, r_m$ of $W_K$ over $\bbQpbar$ such that $(r, N)^\Frss$ is isomorphic to the direct sum $\oplus_{i=1}^m \Sp_{t_i} (r_i)$. So the trace of $r$ is equal to the sum of the traces of the representations
\begin{equation}
\label{Eqn lemma list}
r_i\chi^j , \quad 1\leq i\leq m, 0 \leq j \leq t_i-1,
\end{equation}
which are all irreducible $W_K$-representations. Since the trace of $r$ is also equal to the sum of the traces of the irreducible $W_K$-representations $\sigma_1, \cdots, \sigma_k$, by the Brauer--Nesbitt theorem \cite[30.16]{CurtisReinerReprTheory}, it follows that $\sigma_1, \cdots, \sigma_k$ are isomorphic to the representations in equation \eqref{Eqn lemma list} in some order. Hence $\sigma_1$ is equal to $r_a\chi^b$ for some $a, b$ with $1\leq a\leq m, 0\leq b \leq t_a-1$. Thus $r_a\chi^b$ has maximal weight among the representations in equation \eqref{Eqn lemma list}. So $b$ is necessarily zero. Suppose $(r, N)$ is pure of weight $w$. Then for any $i$, the representation $\Sp_{t_i} (r_i)$ is also pure of weight $w$ by \cite[1.6.7]{DeligneWeil2}. So $r_i$ has weight $w+ (t_i-1)$ for any $i$. Note that $\sigma_1\simeq r_a \chi^b = r_a$ has maximal weight among the representations of equation \eqref{Eqn lemma list}. This shows that the weight of $r_a$ is greater than or equal to the weight of $r_i$ for any $i$, i.e, $w + t_a-1\geq w+ (t_i-1)$ for all $i$. Consequently, the integers $t_a, t_m$ are equal since $t_1\leq \cdots \leq t_m$. Hence we have the following isomorphisms.
\begin{equation}
\label{Eqn: isomorphism}
(r, N)^\Frss 
\simeq \bigoplus_{1\leq i\leq m, i\neq a} \Sp_{t_i} (r_i) \oplus\Sp_{t_a} (r_a)
\simeq \bigoplus_{1\leq i\leq m, i\neq a} \Sp_{t_i} (r_i) \oplus\Sp_{t_m} (\sigma_1)
\end{equation}
Note that $2(t_m-1)$ is the difference of the weights of a highest and a lowest weight representation occurring in equation \eqref{Eqn lemma list}. On the other hand, $2(t-1)$ is equal to the difference of the weights of a highest and a lowest weight representation among $\sigma_1, \cdots, \sigma_k$, which are isomorphic to the representations occurring in equation \eqref{Eqn lemma list} in some order. Hence $t_m$ is equal to $t$. So $(r, N)^\Frss$ is isomorphic to $\oplus_{1\leq i\leq m, i\neq a} \Sp_{t_i} (r_i) \oplus\Sp_{t} (\sigma_1)$. Since $(r, N)^\Frss$ is pure, its direct summand $\oplus_{1\leq i\leq m, i\neq a} \Sp_{t_i} (r_i)$ is also pure by \cite[1.6.7]{DeligneWeil2}. Hence the lemma holds for $(r',N')$ equal to $\oplus_{1\leq i\leq m, i\neq a} \Sp_{t_i} (r_i)$. 
\end{proof}

\section{Pure representations and normalized traces}
\label{Section3}

In this section, we prove the following theorem and use it to deduce Theorem \ref{Theorem Introduction}. 

\begin{theorem}
\label{Theorem WD}
Let $(\rho_1, N_1)$ and $(\rho_2, N_2)$ be pure Weil--Deligne representations of $W_K$ with coefficients in $\bbQpbar$ such that $\rho_1, \rho_2$ have equal normalized traces. Then $(\rho_1, N_1)^\Frss$ is isomorphic to $\psi ^w \otimes (\rho_2, N_2)^\Frss$ where $w$ denotes the difference of the weights of $(\rho_1, N_1)$ and $(\rho_2, N_2)$. 
\end{theorem}

\begin{proof}
Note that the above statement holds for Weil--Deligne representations on one-dimensional spaces. Assume that the underlying spaces of $(\rho_1, N_1), (\rho_2, N_2)$ are of dimension $n$ and the above statement holds for Weil--Deligne representations on spaces of dimension $< n$. By \cite[Proposition 3.1.3(ii)]{DeligneFormesModulairesGL2}, there exist positive integers $t_1\leq  \cdots\leq t_m$ and Frobenius-semisimple representations $r_1, \cdots, r_m$ of $W_K$ with open kernel such that $(\rho_1, N_1)^\Frss$ is isomorphic to the direct sum $\oplus_{i=1}^r \Sp_{t_i} (r_i)$. Since $(\rho_1,N_1)$ is pure, from \cite[1.6.7]{DeligneWeil2}, it follows that the representations $\Sp_{t_i}(r_i)$ are pure of the same weight. Consequently, among the representations
\begin{equation}
\label{Eqn: representation}
\text{$r_i\chi^{j}$ where $1\leq i\leq m, 0 \leq j \leq t_i-1$}
\end{equation}
the $W_K$-representation $r_m$ (resp. $r_m\chi^{t_m-1}$) is a highest (resp. lowest) weight representation. Note that the trace of $\rho_2$ is equal to the trace of $\psi^{-w}\otimes \rho_1$, which is equal to the sum of the traces of the following representations.
\begin{equation}
\label{Eqn: representation twist}
\text{$\psi^{-w}\otimes r_i\chi^{j}$ where $1\leq i\leq m, 0 \leq j \leq t_i-1$}
\end{equation}
The difference of the weights of a highest and a lowest weight representation among these representations is equal to $2(t_m-1)$. Moreover, $\psi^{-w}\otimes r_m$ is a representation of highest weight among them. Since $(\rho_2,N_2)$ is pure, by Lemma \ref{Lemma sum}, $(\rho_2, N_2)^\Frss$ is isomorphic to the direct sum $(\rho, N)\oplus \Sp_{t_m}(\psi^{-w}\otimes r_m)$ for some pure Frobenius-semisimple Weil--Deligne representation $(\rho, N)$. 
Note that by \cite[1.6.7]{DeligneWeil2}, the representation $\oplus_{1\leq i <m}\Sp_{t_i}(r_i)$ (resp. $(\rho, N)$) is pure having weight same as the weight of $(\rho_1, N_1)$ (resp. $(\rho_2, N_2)$). 
Hence the pure Weil--Deligne representations $\oplus_{1\leq i <m}\Sp_{t_i}(r_i)$ and $(\rho, N)$ have equal normalized traces and are of dimensions strictly smaller than $n$. So 
$\oplus_{1\leq i <m}\Sp_{t_i}(\psi^{-w}\otimes r_i)$ and $(\rho, N)$ are isomorphic by the induction hypothesis. Consequently, $(\rho_2, N_2)^\Frss$ is isomorphic to $\oplus_{1\leq i \leq m}\Sp_{t_i}(\psi^{-w}\otimes r_i) \simeq \psi^{-w}\otimes (\rho_1, N_1)^\Frss$. This completes the proof.
\end{proof}

\begin{remark}
\label{Remark}
When $(\rho_1, N_1), (\rho_2, N_2)$ are representations as in the statement of Theorem \ref{Theorem WD}, then the representations $(\rho_1, N_1), (\psi^w\otimes \rho_2, N_2)$ are pure with equal traces. 
Thus their traces arise as specializations of the pseudorepresentation $T:W_K\to \bbQpbar$ given by $T:= \tr \rho_1 =  \tr (\psi^w \otimes \rho_2)$ under the identity map $\mathrm{id}: \bbQpbar \to \bbQpbar$. So Theorem \ref{Theorem WD} can be recovered from \cite[Theorem 4.3]{BigPuritySubAIF}. 
\end{remark}

\begin{proof}
[Proof of Theorem \ref{Theorem Introduction}]
Let $\rho_1, \rho_2$ be as in the statement of Theorem \ref{Theorem Introduction}. Then their Weil--Deligne parametrizations $\WD(\rho_1), \WD(\rho_2)$ are pure with equal normalized traces. By Theorem \ref{Theorem WD}, the Weil--Deligne representation $\WD(\rho_1)^\Frss$ is isomorphic to $\psi^w\otimes \WD(\rho_2)^\Frss$.
\end{proof}

\providecommand{\bysame}{\leavevmode\hbox to3em{\hrulefill}\thinspace}
\providecommand{\MR}{\relax\ifhmode\unskip\space\fi MR }
\providecommand{\MRhref}[2]{%
  \href{http://www.ams.org/mathscinet-getitem?mr=#1}{#2}
}
\providecommand{\href}[2]{#2}


\begin{thebibliography}{KPR16}

\bibitem[CR06]{CurtisReinerReprTheory}
Charles~W. Curtis and Irving Reiner.
\newblock {\em Representation theory of finite groups and associative
  algebras}.
\newblock AMS Chelsea Publishing, Providence, RI, 2006.
\newblock Reprint of the 1962 original.

\bibitem[Del73a]{DeligneFormesModulairesGL2}
P.~Deligne.
\newblock Formes modulaires et repr\'esentations de {${\rm GL}(2)$}.
\newblock In {\em Modular functions of one variable, {II} ({P}roc. {I}nternat.
  {S}ummer {S}chool, {U}niv. {A}ntwerp, {A}ntwerp, 1972)}, pages 55--105.
  Lecture Notes in Math., Vol. 349. Springer, Berlin, 1973.

\bibitem[Del73b]{DeligneConstantesDesEquationsFunctional}
P.~Deligne.
\newblock Les constantes des \'equations fonctionnelles des fonctions {$L$}.
\newblock In {\em Modular functions of one variable, {II} ({P}roc. {I}nternat.
  {S}ummer {S}chool, {U}niv. {A}ntwerp, {A}ntwerp, 1972)}, pages 501--597.
  Lecture Notes in Math., Vol. 349. Springer, Berlin, 1973.

\bibitem[Del80]{DeligneWeil2}
Pierre Deligne.
\newblock La conjecture de {W}eil. {II}.
\newblock {\em Inst. Hautes \'Etudes Sci. Publ. Math.}, (52):137--252, 1980.

\bibitem[Ill94]{IllusieMonodromieLocale}
Luc Illusie.
\newblock Autour du th\'eor\`eme de monodromie locale.
\newblock {\em Ast\'erisque}, (223):9--57, 1994.
\newblock P{\'e}riodes $p$-adiques (Bures-sur-Yvette, 1988).

\bibitem[KPR16]{KulkarniPatankarRajan}
Manisha Kulkarni, Vijay~M. Patankar, and C.~S. Rajan.
\newblock Locally potentially equivalent two dimensional {G}alois
  representations and {F}robenius fields of elliptic curves.
\newblock {\em J. Number Theory}, 164:87--102, 2016.

\bibitem[MP17]{MurtyPujahari}
M.~Ram Murty and Sudhir Pujahari.
\newblock Distinguishing {H}ecke eigenforms.
\newblock {\em Proc. Amer. Math. Soc.}, 145(5):1899--1904, 2017.

\bibitem[PR17]{PatankarRajan}
Vijay~M. Patankar and C.~S. Rajan.
\newblock Distinguishing {G}alois representations by their normalized traces.
\newblock {\em J. Number Theory}, 178:118--125, 2017.

\bibitem[Raj98]{RajanStrongMultOne}
C.~S. Rajan.
\newblock On strong multiplicity one for {$l$}-adic representations.
\newblock {\em Internat. Math. Res. Notices}, (3):161--172, 1998.

\bibitem[Sah14]{Thesis}
Jyoti~Prakash Saha.
\newblock An algebraic $p$-adic {$L$}-function for ordinary families.
\newblock Thesis (Ph.D.)--Universit\'e Paris-Sud, 2014.

\bibitem[Sah17]{BigPuritySubAIF}
Jyoti~Prakash Saha.
\newblock Purity for families of {G}alois representations.
\newblock {\em Ann. Inst. Fourier (Grenoble)}, 67(2):879--910, 2017.

\bibitem[ST68]{SerreTate}
Jean-Pierre Serre and John Tate.
\newblock Good reduction of abelian varieties.
\newblock {\em Ann. of Math. (2)}, 88:492--517, 1968.

\end{thebibliography}
\end{document}